\documentclass[12pt]{amsart}

\usepackage[hmargin=2.5cm,bmargin=2.5cm,tmargin=3cm]{geometry}

\usepackage{amsmath,amssymb,amsthm,amsfonts,enumerate,url}
\usepackage{mathrsfs}
\usepackage{graphicx}

\usepackage[colorlinks, linkcolor=blue, filecolor=blue,
     citecolor = olive, urlcolor=purple]{hyperref}
\usepackage{tikz}
\usepackage{orcidlink}

\usetikzlibrary{automata,positioning,arrows.meta}
\usepackage[utf8]{inputenc}

\theoremstyle{plain}

\numberwithin{equation}{section}
\numberwithin{figure}{section}

\newcommand\Graph{\mathcal{G}}

\newcommand{\C}{\mathbb{C}}
\newcommand{\N}{\mathbb{N}}

\def\dashint{\,\ThisStyle{\ensurestackMath{%
            \stackinset{c}{.2\LMpt}{c}{.5\LMpt}{\SavedStyle-}{\SavedStyle\phantom{\int}}}%
        \setbox0=\hbox{$\SavedStyle\int\,$}\kern-\wd0}\int}

 \def\mGraph{\mathcal{G}} 
 \def\mG{\mathsf{G}} 
 
 \def\mV{\mathsf{V}}
 
 \def\mE{\mathsf{E}}

 \def\me{\mathsf{e}}

\def\Deltadg{\Delta^{\mathrm{F}}_{\mGraph}}\def\Deltang{\Delta^{\mathrm{N}}_{\mGraph}}
\def\Deltanga{\Delta^{\mathrm{N}}_{\mGraph_\alpha}}

\DeclareMathOperator{\diam}{diam}

\DeclareMathOperator{\dist}{dist}

\usepackage{aliascnt}

\theoremstyle{plain}
\newtheorem{theo}{Theorem}[section]

\newaliascnt{cor}{theo}
\newaliascnt{prop}{theo}
\newaliascnt{lemma}{theo}

\newtheorem{lemma}[lemma]{Lemma}
\newtheorem{prop}[prop]{Proposition}

\aliascntresetthe{cor}
\aliascntresetthe{prop}
\aliascntresetthe{lemma}

\theoremstyle{definition} 

\newaliascnt{defi}{theo}
\newaliascnt{assum}{theo}
\newaliascnt{assums}{theo}
\newaliascnt{prob}{theo}
\newaliascnt{conv}{theo}

\newtheorem{defi}[defi]{Definition}

\aliascntresetthe{defi}
\aliascntresetthe{assum}
\aliascntresetthe{assums}
\aliascntresetthe{prob}
\aliascntresetthe{conv}

\theoremstyle{remark}

\newaliascnt{rems}{theo}
\newaliascnt{rem}{theo}
\newaliascnt{exa}{theo}
\newaliascnt{exs}{theo}

\newtheorem{rem}[rem]{Remark}

\aliascntresetthe{rems}
\aliascntresetthe{rem}
\aliascntresetthe{exa}
\aliascntresetthe{exs}

\usepackage[makeroom]{cancel}
\usepackage[normalem]{ulem}

\usepackage{verbatim}
\usepackage{array}

\title[Eigenvalue asymptotics on diagonal combs]{Towards a theory of eigenvalue asymptotics on infinite metric graphs: the case of diagonal combs} 

\author[J.B.~Kennedy]{James B.~Kennedy\orcidlink{0000-0001-5634-0301}}
\author[D.~Mugnolo]{Delio Mugnolo\orcidlink{0000-0001-9405-0874}}
\author[M.~Täufer]{Matthias Täufer\orcidlink{0000-0001-8473-2310}}

\address{James B.~Kennedy, Grupo de F\'isica-Matem\'atica \emph{and} Departamento de Matem\'atica, Faculdade de Ci\^encias da Universidade de Lisboa, Campo Grande, P-1749-016 Lisboa, Portugal}
\email{jbkennedy@fc.ul.pt}

\address{Delio Mugnolo, Lehrgebiet Analysis, Fakultät Mathematik und Informatik, Fern\-Universität in Hagen, D-58084 Hagen, Germany}
\email{delio.mugnolo@fernuni-hagen.de}

\address{Matthias Täufer, Lehrgebiet Analysis, Fakultät Mathematik und Informatik, Fern\-Universität in Hagen, D-58084 Hagen, Germany}
\email{matthias.taeufer@fernuni-hagen.de}



\thanks{The work of J.B.K. was supported by the Funda\c{c}\~ao para a Ci\^encia e a Tecnologia, Portugal, via project PTDC/MAT-PUR/1788/2020 and grant UIDB/00208/2020. The work of D.M.\ was partially supported by the Deutsche Forschungsgemeinschaft (Grant 397230547). This article is based upon work from COST Action 18232 MAT-DYN-NET, supported by COST (European Cooperation in Science and Technology), \url{www.cost.eu}.}

\begin{document}

\begin{abstract}	
We examine diagonal combs, a recently identified class of infinite metric graphs whose properties depend on one parameter. These graphs exhibit a fascinating regime where they possess infinite volume while maintaining purely discrete spectrum for the Neumann Laplacian. In this regime, we establish polynomial upper and lower bounds on the $k$-th eigenvalue, revealing that the eigenvalues grow at a rate strictly slower than quadratic. However, once the diagonal combs transition to finite volume, their growth accelerates to a quadratic rate. Our methodology involves employing spectral geometric principles tailored for metric graphs, complemented by deriving estimates for the $k$-th eigenvalue on compact metric graphs.
\end{abstract}

\maketitle

\section{Introduction}

Metric graphs are metric measure spaces that locally look 1-dimensional, apart from a discrete set $\mV$ of singularities (the \textit{vertices}). 
They arise by identifying each edge $\me\in \mE$ of a classical (combinatorial) graph $\mG=(\mV,\mE)$ with an interval of length $\ell_\me\in (0,\infty)$. They were introduced in the mathematical literature by Lumer in~\cite{Lum80a}. In the 1980s, it was mostly the spectral property of Laplacians that drew the attention of those developing the theory of
metric graphs. 
For this reason, it was natural to focus on metric graphs consisting of finitely many edges \cite{Bel85,Nic85,Ali86}, since in this case the Laplacian is easily seen to have compact resolvent and, hence, purely discrete spectrum. In particular, it has been observed in~\cite{Nic87} that a Weyl-type eigenvalue asymptotics holds in this case, with the leading order term encoding the metric graph's volume: this was first obtained for equilateral graphs, but the latter condition can be removed, see~\cite[Theorem~4.6]{Kur23}.

Infinite metric graphs already appeared in~\cite{Lumer1979-1980}. 
To the best of our knowledge, their spectral properties started being studied in~\cite{AliNic93}. Along with standard vertex conditions, additional ``boundary conditions at infinity'' (more precisely, at the graph ends) may have to be imposed to enforce self-adjointness of the operator \cite{KosMugNic22}. 
The most natural such realizations are the Friedrichs and the Neumann realizations $\Deltadg$ and $\Deltang$, respectively, where Dirichlet- and Neumann-type conditions are imposed at all ends. Their spectral properties in the case of purely discrete spectrum were thoroughly studied in~\cite{DufKenMug22}; here we will focus on the latter.

It was observed in~\cite{Car00} that $\Deltang$ has purely discrete spectrum on all infinite metric graphs of finite volume. 
Notably, the classical proof of Weyl's asymptotics from~\cite{Nic87} breaks down and it seems to be currently unknown whether $\Deltang$ will always have a classical Weyl-type eigenvalue asymptotics on general infinite metric graphs of finite volume.

 Slightly reminiscent of Strauss' embedding theorem for radially symmetric functions, the first nontrivial examples of infinite metric graphs with infinite volume (radially symmetric metric trees with strong volume growth) such that $\Deltadg$ has purely discrete spectrum
have been presented in~\cite{Sol04} and, under less restrictive geometric conditions, in~\cite{KosNic19}. 
A full characterization of infinite metric graphs such that $\Deltadg$ or even $\Deltang$ have purely discrete spectrum was obtained in \cite[Section~3]{DufKenMug22}.

In this note we will delve into a finer spectral analysis of $\Deltang$ and initiate a study of eigenvalue asymptotics on a peculiar class of metric graphs introduced in~\cite{DufKenMug22}: a one-parameter family $(\Graph_\alpha)_{\alpha\in (0,\infty)}$ of infinite, locally finite metric graphs with constant diameter, which we refer to as \textit{diagonal combs}, on which we consider the Neumann Laplacian $\Deltanga$.
It was shown in~\cite{DufKenMug22} that $\Deltanga$ exhibits a phase transition for its spectrum: it is purely discrete for all $\alpha>\frac{1}{2}$ but there is nonempty essential spectrum for all $\alpha\in (0,\frac{1}{2}]$; however, the total volume of $\Graph_\alpha$ is finite if and only if $\alpha > 1$.

As an immediate consequence, clearly, the spectrum cannot have the usual Weyl asymptotics in the regime $\alpha \in (\frac{1}{2},1]$ any more; but it is natural to ask what asymptotics the eigenvalues will satisfy and, in particular, whether we still have at most polynomial growth. The answer is yes: we prove polynomial upper and lower bounds on the $k$-th eigenvalue (with different growths for the upper and the lower bound). More precisely, the main result of this note is that another phase transition exists concerning the eigenvalue asymptotics, corresponding exactly to the critical volume threshold: $\Deltanga$ admits a Weyl asymptotics for all $\alpha>1$ and a non-Weyl asymptotics for $\alpha<1$, while an unusual logarithmic term pops up at the critical parameter $\alpha=1$, see \autoref{theo:1}, \autoref{theo:2} and \autoref{theo:3}, respectively.

We are not aware of results about phase transition for the Weyl asymptotics for any family of Euclidean domains with cusps in higher dimension, let alone for \textit{convex} domains -- a class of domains that is, by~\cite[Remark~4.21]{DufKenMug22}, arguably close to diagonal combs. While (families of) fractals may look like natural candidates for non-Weyl behavior, this topic appears to have been seldom studied in the literature; we are aware of~\cite{AloFre17}, where so-called \textit{Hanoi-attractors} $(K_\alpha)_{\alpha\in (0,\frac{1}{3})}$ are shown \textit{not} to exhibit parameter dependence in the Weyl asymptotics, and in fact the eigenvalues grow polynomially with exponent $\ln 9/\ln 5$ for all $\alpha$.

\section{Definitions and main result}

We will only need to define diagonal comb metric graphs and therefore only provide a lightweight definition of the central objects. 
For more comprehensive definitions of infinite metric graphs, Sobolev spaces and Laplacians on them, we refer the reader to~\cite{Mug19,DufKenMug22} and the references therein.

A metric graph is a (countable) collection of (finite) intervals $(\me)_{\me \in \mE}$, each of length $\ell_\me\in (0,\ell_\me)$, topologically glued at their endpoints to form a graph-like space which carries both a metric structure (the shortest-path metric) as well as a a measure structure (inherited from the Lebesgue measure on the intervals).

All metric graphs that will appear in this note are connected and locally finite, i.e., each vertex is only adjacent to a finite number (indeed, 1, 2, or 3) of further vertices.

\begin{defi}
The \emph{volume} $|\Graph|$ of a metric graph $\mGraph$ is the sum over its edge lengths, $|\mG|:=\sum_{\me\in\mE}\ell_\me$.
\end{defi}

\begin{defi}
The \emph{Sobolev space} $H^1(\mGraph)$ is the space of functions $\phi \colon \mGraph \to \C$ such that
 \begin{itemize}
	\item
	their restriction to each edge $\me$ is in $H^1(\me)$, the Sobolev space of once weakly differentiable $L^2(\me)$ functions with weak derivatives in $L^2(\me)$,	
	\item
	the sum over all $H^1(\me)$-norms of restrictions of $\phi$ to edges is finite, and
	\item
	$f$ is continuous on $\mGraph$.
\end{itemize}
\end{defi}

\begin{defi}
	For a metric graph $\mGraph$, the \emph{Neumann Laplacian} $\Deltang$ is the unique self-adjoint operator associated with the quadratic form
	\begin{equation}
	\label{eq:form}
\phi 
	\mapsto 
	\sum_{\me \in \mE}
	\int_{\me} \lvert \phi'(x) \rvert^2 \mathrm{d} x,\qquad \phi\in H^1(\Graph).
	\end{equation}
	\end{defi}
	This is a nonnegative self-adjoint operator in $L^2(\mGraph)$. If $\Delta_\mGraph$ has purely discrete spectrum, we denote the sequence of its eigenvalues, enumerated increasingly and counting multiplicities, by $\{ \lambda_k(\mGraph)\}_{k \geq 1}$. 

Note that~\eqref{eq:form} is independent of the chosen orientation of an edge due to the square.
Furthermore, is possible to make this definition more general by imposing Dirichlet boundary conditions at vertices and at so-called graph ends of finite volume.
We refer the reader to~\cite{DufKenMug22} for details on this construction.

In this note, we will focus on one class of examples, namely the \emph{diagonal combs}.

\begin{defi}
	For $\alpha > 0$ define the \emph{diagonal comb metric graph} $\mGraph_\alpha$ by taking the interval $(0,1]$ as a ``backbone'' and attaching, for every $n = 1,2,\dots$, an edge of length $\frac{1}{n^\alpha}$ to the point identified with $\frac{1}{n^\alpha} \in (0,1]$. 
	
	More precisely, this is a metric graph the edges of which are the ``teeth'' -- edges of length $\frac{1}{n^{\alpha}}$ -- and edges connecting the bases of the teeth such that the foot point corresponding to the $n$-th tooth is 
	\begin{itemize}
	\item
	connected with the $(n+1)$-st foot point by an edge of length $\frac{1}{n^\alpha} - \frac{1}{(n+1)^\alpha}$,
	\item
	connected with the $(n-1)$-st foot point by an edge of length $\frac{1}{(n-1)^\alpha} - \frac{1}{n^\alpha}$.
	\end{itemize}	
	We refer to Figure~\ref{fig:diagonal_comb} for an illustration.
\end{defi}

\begin{figure}
	\begin{tikzpicture}[scale = 4]
		
		\begin{scope}[xshift = 0cm]
		\draw[thick] (0,0) -- (1,0);
		
		\foreach \x in {1,...,50}
		{
		\draw[thick] ({1/sqrt(\x)},0) -- ({1/sqrt(\x)},{1/sqrt(\x)});
		}
		\draw[fill] (0,0) -- (.15,0) -- (.15,.15) -- (0,0);

		\draw (0.725,-.3) node {\large $n^{-\frac{1}{2}}$};
		\draw (0,-.1) node {$0$};
		\draw (1,-.1) node {$1$};		
		\end{scope}	
	\end{tikzpicture}

\caption{The diagonal comb $\mGraph_\alpha$.}
\label{fig:diagonal_comb}
\end{figure}

Note that the point which would correspond to $0$ on the ``backbone'' is a so-called \textit{graph end}, see~\cite{DufKenMug22, KosNic19} for a definition.
The following is known.

\begin{prop}[{See~\cite{DufKenMug22}, Theorem 3.4 (1)--(2)}]
	The volume $|\mGraph_\alpha|$ is finite if and only if $\alpha \in (1, \infty)$.
	
	The spectrum of the Kirchhoff Laplacian $\Deltanga$ on $\mGraph_\alpha$ is purely discrete if and only if $\alpha \in (\frac{1}{2}, \infty)$, whereas for $\alpha \in (0,\frac{1}{2}]$, there is non-empty essential spectrum.
\end{prop}

So, somewhat counterintuitively, the transition from purely discrete to nonempty essential spectrum does not happen at the finite-to-infinite volume transition, but among graphs of infinite volume.

\begin{rem}\label{rem:automatic-dirichlet}
We observe that $\lambda_1(\Graph_\alpha)>0$ whenever $\alpha \in (\frac{1}{2}, 1]$, since otherwise the associated eigenfunction would be constant and, then, identically 0. Furthermore, in the regime $\alpha \in (\frac{1}{2}, 1]$, the (unique) graph end $\gamma_0$ at the point on the backbone, identified with $0$, is an \emph{end of infinite volume} which implies in particular that $\Deltanga$ will impose a Dirichlet condition at it, see the proof of~\cite[Theorem~3.10]{KosMugNic22}: more precisely, $\lim\limits_{x\to \gamma_0}f(x)=0$, for all $f\in H^1(\Graph_\alpha)$. 
\end{rem}

Our main result sheds light on the asymptotics of the eigenvalues in the regime $\alpha \in (\frac{1}{2}, 1]$.

\begin{theo}
	\label{theo:1}
	For the diagonal comb $\mGraph_\alpha$ with parameter $\alpha \in (\frac{1}{2}, 1)$, there are $c,C > 0$, depending on $\alpha$, such that for sufficiently large $k$
	\[
	c k^{4 \alpha - 2}
	\leq
	\lambda_k(\mGraph_\alpha)
	\leq
	C k^{2 \alpha}.
	\]
\end{theo}

\begin{rem}
We stress that while both bounds differ and we do not know if either of them is sharp, both are polynomial and the upper one rules out quadratic eigenvalue growth as in classical Weyl asymptotics. Indeed, the lower bound interpolates between exponent $2$ and exponent $0$: this suggests that the lower bound might be the optimal one.
We will see in the next section that both bounds rely on surgery principles and delicate eigenvalue estimates on finite metric graphs from~\cite{BerKenKur17}.
\end{rem}

Furthermore, we can infer an upper bound in the boundary case $\alpha = 1$ where we obtain upper and lower bounds that coincide up to a logarithmic term:

\begin{theo}
\label{theo:2}
	For the diagonal comb $\mGraph_\alpha$ with parameter $\alpha = 1$, there are $c,C > 0$ such that for sufficiently large $k$
	\[
	c \frac{k^2}{\ln(k)^4}
	\leq
	\lambda_k(\mGraph_1)
	\leq
	C \frac{k^2}{\ln(k)^2}.
	\]
\end{theo}

The upper bounds in~\autoref{theo:1} and~\autoref{theo:2} are proved in~\autoref{thm:upper_bound}, the lower bounds in~\autoref{thm:lower_bound}. For completeness' sake, we note that for $\alpha$ above the volume transition threshold $\alpha = 1$ we can recover the correct power behaviour in the Weyl asymptotics.

\begin{theo}
	\label{theo:3}
	For the diagonal comb $\mGraph_\alpha$ with parameter $\alpha \in (1,\infty)$, for all $k \geq 1$ we have
\[
\frac{\pi^2 k^2}{4|\Graph_\alpha|^2}
\leq
	\lambda_k(\mGraph_\alpha)
\le 
\frac{\pi^2 k^2}{4}.
	\]
\end{theo}

\newpage
\section{Proofs}

\subsection{Upper bounds}

\begin{theo}
	\label{thm:upper_bound}
	For every $\alpha \in (\frac{1}{2}, 1)$, there is $C = C(\alpha) > 0$ such that for sufficiently large $k$
	\[
	\lambda_k(\mGraph_\alpha)
	\leq
	C k^{2 \alpha}.
	\]
	In the case $\alpha = 1$, there is $C > 0$ such that for sufficiently large $k$
	\[
	\lambda_k(\mGraph_\alpha)
	\leq
	C \frac{k^2}{\ln(k)^2}.
	\]
\end{theo}

We will use the following special case of \cite[Theorem 4.9]{BerKenKur17}.

\begin{theo}
	\label{thm:BKKM_upper}
	Let $\mGraph$ be a connected, compact metric tree graph with Dirichlet or Neumann conditions at all degree one vertices and Kirchhoff--Neumann conditions elsewhere.
Then
	\begin{equation}\label{eq:upper-BKKM1}
	\lambda_k(\mGraph)
	\leq
	\left(
	k - 2 + \lvert D \rvert + \frac{ \lvert N \rvert}{2}
	\right)^2
	\frac{\pi^2}{|\mGraph|^2},
	\end{equation}
	where
	$\lvert D \rvert$ and $\lvert N \rvert$ denote the number of degree one vertices with Dirichlet or Neumann conditions, respectively.
\end{theo}

\begin{rem}
As will become clear in the following proof, the main reason why we resort to the upper bound in~\eqref{eq:upper-BKKM1}, rather than the perhaps simpler estimate in \cite[Theorem~3.2]{Roh17} (which is based on $\lvert \mE \rvert$, the number of edges), is that for the relevant class of truncated combs we are going to study the factor
$\left(
	k - 2 + \lvert D \rvert + \frac{ \lvert N \rvert}{2}
	\right)$ grows linearly in $k$, rather than proportional to $k \lvert \mE \rvert$.
	\end{rem}

\begin{proof}[Proof of \autoref{thm:upper_bound}]
	We cut $\mGraph_\alpha$ in the middle of the edge connecting the bases of the $k$-th largest and the $(k+1)$-st largest tooth. 
	We obtain a finite part $\mGraph_{\alpha,k}$, containing the $k$ largest teeth and an infinite rest.
	Denoting the eigenvalues of the Laplacian on $\mGraph_{\alpha,k}$ with Dirichlet conditions at the cut point by $\lambda_k(\mGraph_{\alpha,k})$ we can certainly estimate
	\[
	\lambda_k(\mGraph_\alpha)
	\leq
	\lambda_k(\mGraph_{\alpha,k}).
	\]
	Using \autoref{thm:BKKM_upper}, we obtain for sufficiently large $k$
	\[
	\lambda_k(\mGraph_{\alpha,k})
	\leq
	\left(
		k + \frac{k - 2}{2}
	\right)^2
	\frac{\pi^2}{|\mGraph_{\alpha,k}|^2}
	\leq
	\frac{9}{4}
	\frac{k^2}{|\mGraph_{\alpha,k}|^2},
	\]
	where 
	\[
	|\mGraph_{\alpha,k}|
	=
	\sum_{l = 1}^k
		l^{- \alpha}	
	+
		1 - \frac{k^{- \alpha} - (k + 1)^{- \alpha}}{2} 
	\]
	is its total volume.
	We estimate, again assuming $k$ sufficiently large,
	\begin{equation}\label{eq:estim-trunc}
	|\mGraph_{\alpha,k}|
	\geq
	\sum_{l = 1}^k
	l^{- \alpha}	
	\geq
	\int_1^k
	x^{-\alpha}
	\mathrm{d} x
	=
	\begin{cases}
	\frac{1}{1 - \alpha}
	\left[
	k^{1 - \alpha}
	-
	1
	\right]
	\geq 
	\frac{1}{2(1 - \alpha)}
	k^{1 - \alpha}
	&
	\text{if $\alpha \in (\frac{1}{2}, 1)$},\\
	\ln(k)
	&
	\text{if $\alpha = 1$}.
	\end{cases}
	\end{equation}
	We conclude 
	\[
	\lambda_k(\mGraph_\alpha)
	\leq
	\begin{cases}
		9 (1 - \alpha)^2 k^{2 \alpha}
		&
		\text{if $\alpha \in (\frac{1}{2}, 1)$},\\
		\frac{9}{4} \frac{k^2}{\ln(k)^2}
		&
		\text{if $\alpha = 1$}.
	\end{cases}
	\qedhere
\]
\end{proof}

\subsection{Lower bounds}

We turn to lower bounds and first cite the following consequence of \cite[Theorem~4.7]{BerKenKur17}.

\begin{theo}
	\label{thm:BKKM_lower}
	Let $\mGraph$ be a compact metric tree graph with at least one Dirichlet vertex.
	Then, 
	\[
	\lambda_k(\mGraph)
	\geq
		\frac{k^2 \pi^2}{4 |\mGraph|^2}
	\quad
	\text{for all $k \geq 1$}.
	\]
\end{theo}

We prove some preparatory lemmas.
For $n \in \N$, we split $\mGraph_\alpha$ into two graphs:
\begin{itemize}
	\item
	An infinite, connected part $\mGraph_{\alpha, n, \infty}$, containing all teeth starting from the $n$-th largest, and the pieces of the backbone, connecting these teeth.
	\item
	A finite, connected part, $\mGraph_{\alpha, 0, n-1}$ consisting of the $(n - 1)$ largest teeth plus the part of the backbone connecting the base of the $n$-th  with the $(n-1)$-st largest tooth.
\end{itemize}

\begin{lemma}
	\label{lem:lower_1}
	Let $\Delta_{\alpha, n, \infty}$ and $\Delta_{\alpha, 0, n-1}$ denote the Laplacians with Kirchhoff--Neumann conditions at the cut point, respectively.
	If
	\[
	\inf \sigma(\Delta_{\alpha, n, \infty}) \geq \lambda_k(\Delta_{\alpha, 0, n - 1}),
	\] 
	then
	\[
	\lambda_k(\Delta_\alpha)
	\geq
	\lambda_k(\Delta_{\alpha, 0, n - 1}).
	\]
\end{lemma}

\begin{proof}
	Cutting with Kirchhoff--Neumann conditions decreases the spectrum.
	Hence, the $k$-th eigenvalue of $\Delta_\alpha$ is bounded from below by the $k$-th smallest element of the set
	\[
	\sigma(\Delta_{\alpha, n, \infty}) \cup \sigma(\Delta_{\alpha, 0, n - 1}),
	\]
	counting multiplicities.
	But by assumption, this is $\lambda_k(\Delta_{\alpha, 0, n - 1}$).
\end{proof}

\begin{lemma}
	\label{lem:lower_2}
	For $\alpha \in (\frac{1}{2}, 1]$, we have
	\[
	\inf \sigma (\Delta_{\alpha, n, \infty})
	\geq
	\frac{2 \alpha - 1}{2}
	\cdot
	n^{2 \alpha - 1}.
	\]
\end{lemma}

\begin{proof}	
By the Courant variational principle,
\[
	\inf \sigma (\Delta_{\alpha, n, \infty})
	=
	\inf_{\phi \in H^1(\mGraph),\ \phi \neq 0}
	\frac{\lVert \phi' \rVert_{L^2(\mGraph)}^2}{\lVert \phi \rVert_{L^2(\mGraph)}^2}.
\]
Take $x \in \mGraph_{\alpha, n, \infty}$ at distance $t > 0$ to the end, and pick the unique path connecting the end to $t$.
For $\phi \in H^1(\mGraph)$, by the fundamental theorem of calculus for $H^1$ functions and the Cauchy-Schwarz inequality,
\[
	\lvert \phi(x) \rvert
	\leq
	\int_0^t
	\lvert
	\phi'(t)
	\rvert
	\mathrm{d} t
	\leq
	\sqrt{t}
	\lVert \phi' \rVert_{L^2(\mGraph)}.
\]
Thus,
\[
	\lVert \phi \rVert_{L^2(\mGraph)}^2
	\leq
	\lVert \phi' \rVert_{L^2(\mGraph)}^2
	\int_{\mGraph_{\alpha, n, \infty}}
	\dist (x, 0)
	\
	\mathrm{d} x
\]
or, equivalently,
\[
	\inf \sigma (\Delta_{\alpha, n, \infty})
	\geq
	\left(
	\int_{\mGraph_{\alpha, n, \infty}}
	\dist (x, 0)
	\
	\mathrm{d} x
	\right)^{-1}.
\]
We calculate
\begin{align*}
	\int_{\mGraph_{\alpha, n, \infty}}
	\dist (x, 0)
	\
	\mathrm{d} x
	&=
	\int_0^{n^{-\alpha}} x\ \mathrm{d} x
	+
	\sum_{l = n}^\infty
	\int_{l^{-\alpha}}^{2 l^{-\alpha}}
	x\
	\mathrm{d} x
	\\
	&=
	\frac{1}{2}
	\left[
		n^{- 2 \alpha}
		+
		3
		\sum_{l = n}^\infty
		l^{- 2 \alpha}
	\right]
	\leq
	2
	\sum_{l = n}^\infty l^{-2 \alpha}
	\\
	&\leq
	2
	\int_{n - 1}^\infty x^{-2 \alpha}\ \mathrm{d} x
	=
	\frac{1}{2 \alpha - 1} (n - 1)^{1 - 2 \alpha}.
	\qedhere
\end{align*}
\end{proof}

\begin{lemma}
	\label{lem:lower_3}
	There is a constant $C = C(\alpha) > 0$ such that for sufficiently large $k,n$ we have
	\[
	\lambda_k(\mGraph_{\alpha, 0,n-1})
	\leq
	\begin{cases}
	C
	\frac{k^2}{n^{2 - 2 \alpha}}
	&
	\text{if $\alpha \in (\frac{1}{2}, 1)$, and}
	\\
	C \frac{k^2}{\ln(n)^2}
	&
	\text{if $\alpha = 1$}.
	\end{cases}	
	\]
\end{lemma}

\begin{proof}
	For $\alpha \in (\frac{1}{2},1]$ we have
	\[
	|\mGraph_{\alpha, 0,n-1}|
	=
	(1 - n^{- \alpha})
	+
	\sum_{l = 1}^{n-1}
	n^{- \alpha}
	\geq
	\int_1^n
	x^{- \alpha}
	\
	\mathrm{d} x
	= \begin{cases}
	\frac{1}{1 - \alpha}
	\left[ n^{1 - \alpha} - 1 \right] & \text{if $\alpha \in (\frac{1}{2}, 1)$,}\\
	\ln (n) & \text{if $\alpha = 1$.}
	\end{cases}
	\]
	We now use these bounds on the estimate on $\lambda_k$ from \autoref{thm:BKKM_upper}:
	\[
	\lambda_k(\mGraph_{\alpha, 0,n-1})
	\leq
	\left(k - 2 + \frac{k}{2} \right)^2
	\frac{\pi^2}{|\mGraph_{\alpha, 0,n-1}|^2}
	\leq
	\begin{cases}
	C
	\frac{k^2}{n^{2 - 2 \alpha}}
	&
	\text{if $\alpha \in (\frac{1}{2}, 1)$, and}\\
	C \frac{k^2}{\ln(n)^2}
	&
	\text{if $\alpha = 1$.}
	\end{cases}
	\qedhere	
	\]
\end{proof}

\begin{theo}
	\label{thm:lower_bound}
	For every $\alpha \in (\frac{1}{2}, 1)$, then there is $c = c(\alpha) > 0$ such that for sufficiently large $k$
	\[
	\lambda_k(\mGraph_\alpha)
	\geq
	c
	k^{4 \alpha-2}.
	\]
	If $\alpha = 1$, then there is $c > 0$ such that for sufficiently large $k$
	\[	
	\lambda_k(\mGraph_\alpha)
	\geq
	c \frac{k^2}{\ln(k)^4}.
	\]
\end{theo}

\begin{proof}
	In the following proof, $C>0$ will denote a constant, possibly depending on $\alpha$ but not on $k$ or $n$, which may change from line to line.
	We first consider the case $\alpha \in (\frac{1}{2}, 1)$.	
	From~\autoref{lem:lower_2} and~\autoref{lem:lower_3}, we infer 
	\begin{equation}
	\label{eq:estimate_infimum}
	\inf \sigma(\Delta_{\alpha, n, \infty})
	\geq
	\lambda_k(\Delta_{\alpha, 0, n-1})
	\end{equation}
	as soon as
	\[
	\frac{2 \alpha - 1}{2}
	\cdot
	n^{2 \alpha - 1}
	\geq
	C
	\frac{k^2}{n^{2 - 2 \alpha}}
	\]
	which is the case if $n \geq C k^2$.
	Consequently, with this choice, the conditions of ~\autoref{lem:lower_1} are satisfied.
	\autoref{thm:BKKM_lower} then yields
	\[
	\lambda_k(\mGraph_\alpha)
	\geq
	\frac{k^2 \pi^2}{4 |\mGraph_{\alpha, 0, \lceil C k^2 \rceil - 1}|^2}.
	\]
	Estimating
	\[
	|\mGraph_{\alpha, 0, \lceil C k^2 \rceil - 1}|
	\leq
	1
	+
	\sum_{l = 1}^{\lceil C k^2 \rceil - 1}
	l^{- \alpha}
	\leq
	1
	+
	\int_1^{C k^2}
	x^{- \alpha}\ \mathrm{d}	x
	=
	1
	+
	\frac{1}{1 - \alpha}
	\left[
	C^{1 - \alpha} k^{2 - 2 \alpha} - 1
	\right]
	\leq
	C
	k^{2 - 2 \alpha}
	\]
	leads to
	\[
	\lambda_k(\mGraph_\alpha)
	\geq
	c k^{-2 + 4 \alpha}
	\]
	for a suitable constant $c>0$.

	In the case $\alpha = 1$, we proceed analogously and use \autoref{lem:lower_2} as well as~\autoref{lem:lower_3} to infer ~\eqref{eq:estimate_infimum}, as soon as
	\[
	n \ln(n)^2 \geq C k^2.
	\]
	Denoting by $\eta \colon [0, \infty) \to [1, \infty)$ the inverse of $x \mapsto x \cdot \ln(x)^2$, we can proceed as before: the conditions of ~\autoref{lem:lower_1} are satisfied and
	\autoref{thm:BKKM_lower} yields
	\[
	\lambda_k(\mGraph_1)
	\geq
	\frac{k^2 \pi^2}{4 |\mGraph_{1, 0, \lceil \eta( C k^2) \rceil - 1}|^2}.
	\]
	We estimate, again assuming $k$ sufficiently large,
	\begin{align*}
	|\mGraph_{1, 0, \lceil \eta( C k^2) \rceil - 1}|
	&\leq
	1
	+
	\sum_{l = 1}^{ \lceil \eta( C k^2) \rceil - 1} l^{-1}
	\leq
	1
	+
	\int_{1}^{\eta(C k^2)}
	x^{-1} \mathrm{d} x
	\\	
	&=
	1 + \ln(\eta(C k^2))
	\leq
	C \ln (\eta(C k^2))
	=
	C
	\frac{C k^2}{\eta(C k^2) \ln(\eta(C k^2))}
	\leq
	\frac{C k^2}{\eta(C k^2)},
	\end{align*} 
	and conclude
	\[
	\lambda_k(\mGraph_1)
	\geq
	C
	\frac{\eta(C k^2)^2}{k^2}.
	\]
	Using $\eta(x) \geq \frac{x}{\ln(x)^2}$ for large $x$ and since $k$ is assumed to be sufficiently large, we can bound this from below by
	\[
	\lambda_k(\mGraph_1)
	\geq
	C \frac{k^2}{\ln(C k^2)^4}
	=
	C
	\frac{k^2}{(\ln(C) + 2 \ln(k))^4}
	\geq
	c \frac{k^2}{\ln(k)^4}
	\]
	for a suitable constant $c=c(\alpha)>0$.	
\end{proof}

\subsection{Weyl asymptotics for finite volume comb graphs}

\begin{proof}[Proof of \autoref{theo:3}]

The lower bound holds by \cite[Theorem~4.1]{DufKenMug22}. For the upper bound, consider the metric tree $\Graph_{\alpha,k}$ obtained truncating $\Graph_{\alpha}$ at the midpoint between the $k$-th and $(k+1)$-th teeth, like in the proof of \autoref{thm:upper_bound}, and impose a Dirichet condition at the cut point.
We denote the resulting operator by $\Delta^{\mathrm{N,D}}_{\Graph_{\alpha,k}}$, wheres $\Delta^{\mathrm{N,N}}_{\Graph_{\alpha,k}}$ denotes the operator obtained imposing a Neumann condition at the cut point.

By domain monotonicity, the $k$-th eigenvalue of $\Delta^{\mathrm{N}}_{\Graph_\alpha}$ is bounded from above by the $k$-th eigenvalue of $\Delta^{\mathrm{N,D}}_{\Graph_{\alpha,k}}$ which is, by \cite[Theorem~3.4]{BerKenKur19}, not larger than $(k+1)$-th eigenvalue of $\Delta^{\mathrm{N,N}}_{\Graph_{\alpha,k}}$. The latter can, in turn, be controlled using \cite[Theorem~3.4]{Roh17} 
As $\diam(\Graph_{\alpha,k})=2$ for all $k\ge 2$ we conclude that
\[
\frac{k^2 \pi^2}{4|\Graph_\alpha|^2}\le 
\lambda_k(\Graph_\alpha)\le \frac{k^2 \pi^2}{4}.\qedhere
\]
\end{proof}

\begin{rem}
For $\alpha >1$ one also has the asymptotics
\begin{equation}\label{eq:upper-asympt-a>1}
	\limsup_{k \to \infty} \frac{\lambda_k (\Graph_\alpha)}{k^2} \leq \frac{\pi^2}{|\Graph_\alpha|^2}
\end{equation}
as follows: consider the graph $\Graph_\alpha^D$ which has the same edge set as $\Graph_\alpha$, but Dirichlet conditions at all vertices; then $\Graph_\alpha^D$ is equivalent to a disjoint union of (infinitely many) Dirichlet intervals of the same total volume as $\Graph_\alpha$. But since clearly $H^1_0(\Graph_\alpha^D)$ can be canonically identified with a subset of $H^1(\Graph_\alpha)$, if $\lambda_k (\Graph_\alpha^D)$ denotes the $k$-th eigenvalue of the Dirichlet Laplacian on $\Graph_\alpha^D$, then the usual min-max characterisation implies that $\lambda_k (\Graph_\alpha) \leq \lambda_k^D (\Graph_\alpha)$ for all $k$. Now an elementary calculation involving the eigenvalue counting function of $\Graph_\alpha^D$ (which is the sum of the individual eigenvalue counting functions of the constituent intervals) shows that $\lambda_k (\Graph_\alpha^D) = \frac{\pi^2 k^2}{|\Graph_\alpha^D|^2} + o(k^2)$, from which \eqref{eq:upper-asympt-a>1} follows.
\end{rem}

\end{document}